\documentclass[a4paper,11pt]{article}
\usepackage{graphicx}
\usepackage{amsfonts}
\usepackage{amsmath}
\usepackage{amssymb}
\usepackage{amsthm}
\usepackage{hyperref}
\usepackage{xcolor}

\title{Special orientable sequences}
\author{Chris J. Mitchell and Peter R. Wild
\\Information Security Group, Royal Holloway, University of London\\
\href{mailto:me@chrismitchell.net}{me@chrismitchell.net};
~~~~\href{mailto:peterrwild@gmail.com}{peterrwild@gmail.com}}
\date{26th November 2024}
\newtheorem{lemma}{Lemma}[section]
\newtheorem{theorem}[lemma]{Theorem}
\newtheorem{corollary}[lemma]{Corollary}
\newtheorem{definition}[lemma]{Definition}
\newtheorem{remark}[lemma]{Remark}
\newtheorem{construction}[lemma]{Construction}

\newtheorem{example}[lemma]{Example}

\begin{document}

\maketitle

\begin{abstract}
Analogously to de Bruijn sequences, Orientable sequences have application in automatic
position-location applications and, until recently, studies of these sequences focused on the
binary case. In recent work by Alhakim et al., recursive methods of construction were described for
orientable sequences over arbitrary finite alphabets, requiring `starter sequences' with special
properties. Some of these methods required as input special orientable sequences, i.e.\ orientable
sequences which were simultaneously negative orientable.  We exhibit methods for constructing
special orientable sequences with properties appropriate for use in two of the recursive methods of
Alhakim et al.  As a result we are able to show how to construct special orientable sequences for
arbitrary sizes of alphabet (larger than a small lower bound) and for all window sizes.  These
sequences have periods asymptotic to the optimal as the alphabet size increases.
\end{abstract}

\section{Introduction}

Orientable sequences, i.e.\ periodic sequences with elements drawn from a finite alphabet with the
property that any subsequence of $n$ consecutive elements (an $n$-tuple) occurs at most once
\emph{in either direction}, were introduced in 1992 \cite{Burns92,Dai93}. They are of interest due
to their application in certain position-resolution scenarios. For the binary case, a construction
and an upper bound on the period were established by Dai et al.\ \cite{Dai93}, and further
constructions were established by Gabri\'{c} and Sawada \cite{Gabric24} and Mitchell and Wild
\cite{Mitchell22}. A bound on the period and methods of construction for $q$-ary alphabet sequences
(for arbitrary $q$) were given by Alhakim et al.\ \cite{Alhakim24a}.

In this paper we examine a particular class of orientable sequences known as \emph{special
orientable sequences}; such sequences were defined by Alhakim et al.\ \cite{Alhakim24a}, who
described a series of recursive constructions for orientable sequences using special orientable
sequences as input. We give a bound on the length of special orientable sequences and describe
various methods of construction. We then show how certain of the constructed sequences can be used
to obtain orientable sequences using methods defined in \cite{Alhakim24a}.

\subsection{Basic terminology}

We first establish some simple notation, following \cite{Alhakim24a}. For mathematical convenience
we consider the elements of a sequence to be elements of $\mathbb{Z}_q$ for an arbitrary integer
$q>1$.

For a sequence $S = (s_i)$ we write $\mathbf{s}_n(i) = (s_i,s_{i+1},\ldots,s_{i+n-1})$.  Since we
are interested in tuples occurring either forwards or backwards in a sequence we also introduce the
notion of a reversed tuple, so that if $\mathbf{u} = (u_0,u_1,\ldots,u_{n-1})$ is a $q$-ary
$n$-tuple (a string of symbols of length $n$) then $\mathbf{u}^R = (u_{n-1},u_{n-2}, \ldots,u_0)$
is its \emph{reverse}.  The \emph{negative} of a $q$-ary $n$-tuple
$\mathbf{u}=(u_0,u_1,\ldots,u_{n-1})$ is the $n$-tuple $-\mathbf{u}=(-u_0,-u_1,\ldots,-u_{n-1})$.

We can then give the following.

\begin{definition}[\cite{Alhakim24a}]
A $q$-ary \emph{$n$-window sequence $S = (s_i)$} is a periodic sequence of elements from
$\mathbb{Z}_q$ ($q>1$, $n>1$) with the property that no $n$-tuple appears more than once in a
period of the sequence, i.e.\ with the property that if $\mathbf{s}_n(i) = \mathbf{s}_n(j)$ for
some $i,j$, then $i \equiv j \pmod m$ where $m$ is the period of the sequence.
\end{definition}

\begin{definition}[\cite{Alhakim24a}]
An $n$-window sequence $S = (s_i)$ is said to be an orientable sequence of order $n$ (an
$\mathcal{OS}_q(n)$) if $\mathbf{s}_n(i) \neq \mathbf{s}_n(j)^R$, for any $i,j$.
\end{definition}

We also need two related definitions.

\begin{definition}[\cite{Alhakim24a}]  \label{definition:NOS}
An $n$-window sequence $S=(s_i)$ is said to be a \emph{negative orientable sequence of order $n$}
(a $\mathcal{NOS}_q(n)$) if $\mathbf{s}_n(i)\not=-{\mathbf{s}_n(j)}^R$, for any $i,j$.
\end{definition}

\begin{definition}[\cite{Alhakim24a}]  \label{definition:specialOS}
An orientable sequence $S=(s_i)$ of order $n$ is said to be a \emph{special orientable sequence of
order $n$} (a $\mathcal{SOS}_q(n)$)  if, for any $i,j$,
$\mathbf{{s}}_n(i)\not={-\mathbf{s}}_n(j)^R$, i.e.\ it is also negative orientable.
\end{definition}

As discussed in Alhalkim at al.\ \cite{Alhakim24a}, it turns out that negative and special
orientable sequences are of importance in constructing orientable sequences. Observe that a
sequence is orientable if and only if it is negative orientable for the case $q=2$. Also note that
if $S=(s_i)$ is orientable, negative orientable or special then so is its negative ($-s_i$).

Bounds on the period of, and methods of construction for, negative orientable
sequences were given by Mitchell and Wild \cite{Mitchell24}; they also showed
how to use the constructed negative orientable sequences to construct families
of orientable sequences employing two approaches defined in \cite{Alhakim24a}.
By contrast, in this paper we focus on special orientable sequences, giving a
period bound and methods of construction.

\subsection{The de Bruijn graph and the Lempel Homomorphism}  \label{subsection:dB_graph_and_Lempel}

Following Alhakim et al.\ \cite{Alhakim24a} we also introduce the de Bruijn graph. For positive
integers $n$ and $q$ greater than one, let $\mathbb{Z}_q^n$ be the set of all $q^n$ vectors of
length $n$ with entries from the group $\mathbb{Z}_q$ of residues modulo $q$. A de Bruijn sequence
\emph{of order $n$} with alphabet in $\mathbb{Z}_q$ is a periodic sequence that includes every
possible $n$-tuple precisely once as a subsequence of consecutive symbols in one period of the
sequence.

The order $n$ de Bruijn digraph, $B_n(q)$, is a directed graph with $\mathbb{Z}^n_q$ as its vertex
set and where, for any two vectors $\textbf{x} = (x_1,x_2,\ldots,x_n)$ and $\textbf{y} =
(y_1,y_2,\ldots,y_n), \ (\textbf{x}; \textbf{y})$ is an edge if and only if $y_i = x_{i+1}$ for
every $i$ ($1\leq i< n$). We then say that $\textbf{x}$ is a \emph{predecessor} of $\textbf{y}$ and
$\textbf{y}$ is a \emph{successor} of $\textbf{x}$. Evidently, every vertex has exactly $q$
successors and $q$ predecessors. Furthermore, two vertices are said to be \emph{conjugates} if they
have the same set of successors.

A cycle in $B_n(q)$ is a path that starts and ends at the same vertex. It is said to be
\emph{vertex disjoint} if it does not visit any vertex more than once. Two cycles or two paths in
the digraph are vertex-disjoint if they do not have a common vertex.

Following the notation of Lempel \cite{Lempel70}, a convenient representation of a vertex disjoint
cycle $(\textbf{x}^{(1)}; \dots ; \textbf{x}^{(l)})$ is the \emph{ring sequence} $[x^1, \dots , x^l
]$ of symbols from $\mathbb{Z}_q$ defined such that the $i$th vertex in the cycle starts with the
symbol $x^i$. Corresponding to the ring sequence $[x^1, \dots , x^l]$ is an $n$-window sequence
$S=(s_i)$ where $s_{i+tl}=x_{i+1}$ for $i=0,\dots,l-1$ and $t\ge 0$. Conversely, an $n$-window
sequence determines a ring sequence of a vertex disjoint cycle. A \emph{translate} of a word
$\textbf{x} = (x_1, \dots , x_n)$ is a word $\textbf{x}+\lambda = (x_1 + \lambda, \dots , x_n +
\lambda)$ where $\lambda$ is any nonzero element in $\mathbb{Z}_q$ and addition is performed in
$\mathbb{Z}_q$. We also define a translate of a cycle as the cycle obtained by a translate of the
ring sequence that defines this cycle.

Finally, we need a well-established generalisation of the Lempel graph homomorphism \cite{Lempel70}
to non-binary alphabets --- see, for example, Alhakim and Akinwande \cite{Alhakim11}.

\begin{definition} \label{Lempel}
For a nonzero $\beta \in \mathbb{Z}_q$, we define a function $D_{\beta}$ from $B_n(q)$ to
$B_{n-1}(q)$ as follows. For $a = (a_1, \dots , a_n)$ and $b = (b_1, \dots , b_{n-1}), \
D_{\beta}(a) = b$ if and only if $b_i = d_{\beta}(a_i , a_{i+1})$ for $i = 1$ to $n-1$, where
$d_{\beta}(a_i , a_{i+1}) = \beta(a_{i+1} - a_i) \mod q.$
\end{definition}

We extend the notation to allow the Lempel morphism $D_{\beta}$ to be applied to periodic sequences
in the natural way, as we now describe. That is, $D_{\beta}$ (where $\beta\in \mathbb{Z}_q$) is the
map from the set of periodic sequences to itself defined by
\[ D((s_i))= \{(t_i): t_j=\beta(s_{j+1}-s_{j}) \}. \]
The image of a sequence of period $m$ will clearly have period dividing $m$. In the usual way we
can define $D_{\beta}^{-1}$ to be the \emph{inverse} of $D_{\beta}$, i.e.\ if $S$ is a periodic
sequence than $D_{\beta}^{-1}(S)$ is the set of all sequences $T$ with the property that
$D_{\beta}(T) = S$.

We are particular interested in the case $\beta=1$, and we simply write $D$ for $D_1$. The
\emph{weight} $w(S)$ of a  sequence  $S$ is the weight of the ring sequence corresponding to $S$ (that is the sum of the terms $s_0,\dots,s_{m-1}$ treating $s_i$ as an integer in the range $[0,q-1]$).
Similarly we write $w_q(S)$ for $w(S) \bmod{q}$.

\subsection{Related work}

This paper builds on the work of Alhakim et al.\ \cite{Alhakim24a}, in which recursive methods of
construction for non-binary orientable sequences are described. Alhakim et al.\ described a range
of methods of recursively generating orientable sequences using sequences with special properties,
notably negative orientable and special orientable sequences. However, general methods for
providing `starter' sequences for these constructions were not provided, and this paper is aimed at
addressing this.

In a recent paper \cite{Gabric24a}, Gabri\'{c} and Sawada showed how to construct non-binary
orientable sequences of asymptotically maximal period.  Their approach involves applying the
inverse Lempel Homomorphism  to an orientable sequence and then demonstrating ways to join together
the multiple sequences that result. In parallel work, Mitchell and Wild \cite{Mitchell24} showed
how to construct orientable sequences using a rather different approach, namely first constructing
negative orientable sequences and then applying certain methods of Alhakim et al.\ to construct
larger period orientable sequences.  This paper follows a similar path, except that we show how to
construct special orientable sequences, and then use these in other methods of Alhakim et al.\ to
construct larger period orientable sequences.

\section{A simple period bound}

By definition it follows automatically that the period of an $\mathcal{SOS}_{q}(n)$ is bounded
above by the bounds on the period of an orientable sequence and that of a negative orientable
sequence established in \cite[Theorem 4.11]{Alhakim24a} and \cite[Theorem 3.10]{Mitchell24}. We
next give a bound on the period of an $\mathcal{SOS}_{q}(n)$ which is of the same order as these
general bounds.


\begin{theorem} \label{theorem:SOS_bound}
Suppose $S$ is an $\mathcal{SOS}_{q}(n)$. Then the period of $S$ is at most:
\begin{align*}
\frac{q^n-q^{(n+1)/2}-q^{(n-1)/2}+1}{2} ~~&~~\mbox{if $q$ and $n$ are both odd;}\\
\frac{q^n-2q^{n/2}+1}{2} ~~&~~\mbox{if $q$ is odd and $n$ is even;}\\
\frac{q^n-q^{(n+1)/2}-2q^{(n-1)/2}+2^{(n+3)/2}-2^{(n+1)/2}}{2} ~~&~~\mbox{if $q$ is even and $n$ is odd;}\\
\frac{q^n-2q^{n/2}+2^{(n+2)/2}-2^{n/2}}{2} ~~&~~\mbox{if $q$ and $n$ are both even.}
\end{align*}
\end{theorem}

\begin{proof}
First observe that if an $n$-tuple {\bf s} satisfies $\mathbf s=\mathbf s^R$ or $\mathbf s=-\mathbf
s^R$ then it cannot occur in $S$ since $S$ is both orientable and negative orientable.  Hence,
since at most one of $\mathbf s$ and $\mathbf s^R$ can occur in $S$, the period of $S$ is at most
half the number of $q$-ary $n$-tuples $s$ such that $\mathbf s\not=\mathbf s^R$ and $\mathbf
s\not=-\mathbf s^R$. We examine the four cases separately. Note that the $q$ odd cases are simpler,
since when $q$ is odd there is only one $n$-tuple satisfying $\mathbf s=-\mathbf s$, namely the
all-zero $n$-tuple.

\begin{itemize}
\item Suppose $q$ and $n$ are both odd. Then there is one $n$-tuple {\bf s} with $\mathbf
    s=-\mathbf s=\mathbf s^R=-\mathbf s^R$; $q^{(n+1)/2}-1$ tuples with $\mathbf s=\mathbf
    s^R\not=-\mathbf s=-\mathbf s^R$; $q^{(n-1)/2}-1$ with $\mathbf s=-\mathbf s^R\not=-\mathbf
    s=\mathbf s^R$; and hence $D=q^n-q^{(n+1)/2}-q^{(n-1)/2}+1$ with $\mathbf s,-\mathbf
    s,\mathbf s^R,-\mathbf s^R$ all distinct. The bound follows ($=D/2$).

\item Suppose $q$ is odd and $n$ is even. Then there is one $n$-tuple {\bf s} with $\mathbf
    s=-\mathbf s=\mathbf s^R=-\mathbf s^R$; $q^{n/2}-1$ with $\mathbf s=\mathbf
    s^R\not=-\mathbf s=-\mathbf s^R$; $q^{n/2}-1$ with $\mathbf s=-\mathbf s^R\not=-\mathbf
    s=\mathbf s^R$; and hence $D=q^n-2q^{n/2}+1$ with $\mathbf s,-\mathbf s,\mathbf
    s^R,-\mathbf s^R$ all distinct. The bound follows ($=D/2$).

 \item Suppose $q$ is even and $n$ is odd. Then there are~%
   $2^{(n+1)/2}$ $n$-tuples {\bf s} with $\mathbf s=-\mathbf s=\mathbf s^R=-\mathbf s^R$;
   $N=2^n-2^{(n+1)/2}$ $n$-tuples with $\mathbf s=-\mathbf s \not=\mathbf s^R=-\mathbf s^R$;
   $q^{(n+1)/2}-2^{(n+1)/2}$ with $\mathbf s=\mathbf s^R \not=-\mathbf s=-\mathbf s^R$;
   $2q^{(n-1)/2}-2^{(n+1)/2}$ with $\mathbf s=-\mathbf s^R \not=-\mathbf s=\mathbf s^R$; and
   hence~%
   $D=q^n-q^{(n+1)/2}-2q^{(n-1)/2}+2^{(n+3)/2}-2^n$ with $\mathbf s,-\mathbf s,\mathbf
    s^R,-\mathbf s^R$ all distinct. The bound is $(N+D)/2$, and the result follows.

 \item Suppose $q$ and $n$ are both even. Then there are~%
  $2^{n/2}$ $n$-tuples {\bf s} with $\mathbf s=-\mathbf s=\mathbf s^R=-\mathbf s^R$;
     $N=2^n-2^{n/2}$ $n$-tuples with $\mathbf s=-\mathbf s\not=\mathbf s^R=-\mathbf s^R$;
     $q^{n/2}-2^{n/2}$ with $\mathbf s=\mathbf s^R\not=-\mathbf s=-\mathbf s^R$;
     $q^{n/2}-2^{n/2}$ with $\mathbf s=-\mathbf s^R\not=-\mathbf s=\mathbf s^R$; and hence~%
     $D=q^n-2q^{n/2}+2^{(n+2)/2}-2^n$ with $\mathbf s,-\mathbf s,\mathbf s^R,-\mathbf s^R$ all
     distinct. The bound is $(N+D)/2$, and the result follows.
\end{itemize}
\end{proof}

\section{Constructing special orientable sequences}

\subsection{A simple construction}

We first show how to construct an $\mathcal{SOS}_{q}(n)$ with period about one quarter the bound
given by Theorem~\ref{theorem:SOS_bound} for every odd $q \ge 5$ when $n=2$.

\begin{construction}  \label{construction:qprime}
Let $q,q'$ be integers with $q' > q >1$. For $x \in \mathbb{Z}_q$ we write $\underline x$ for the
non-negative integer in $\{0,1,\dots,q-1\}$ belonging to the residue class $x$, and $x'$ for the
residue class of $\mathbb Z_{q'}$ that contains $\underline x$. Let $S=[s_0,\dots,s_{m-1}]$ be an
$\mathcal{OS}_q(n)$. Let $S'=[s_0^{\prime},\dots,s_{m-1}^{\prime}]$ be the sequence over $\mathbb
Z_{q'}$ obtained from $S$ in the obvious notational way.
\end{construction}

\begin{theorem}
If $S$ is an $\mathcal{OS}_q(n)$, $q' \ge 2q-1$ and $S'$ is obtained from $S$ using
Construction~\ref{construction:qprime}, then $S'$ is an $\mathcal{SOS}_{q'}(n)$.
\end{theorem}

\begin{proof}

First observe that if $x$ is a non-zero term of $S'$ then $-x\not=y$ for any term $y$ of $S'$.

Suppose $0\leq i,j<m$.  We need to consider three cases.
\begin{itemize}
\item Suppose $\mathbf{s}_n^{\prime}(i)=\mathbf{s}_n^{\prime}(j)$. Then
    $\mathbf{s}_n(i)=\mathbf{s}_n(j)$ and so $i\equiv j \pmod m$ (i.e.\ $i=j$).
\item Suppose $\mathbf{s}_n^{\prime}(i)=\mathbf{s}_n^{\prime R}(j)$. Then
    $\mathbf{s}_n(i)=\mathbf{s}_n^R(j)$. This is impossible since $S$ is an
    $\mathcal{OS}_q(n)$.
\item Finally, suppose $\mathbf{s}_n^{\prime}(i)=-\mathbf{s}_n^{\prime R}(j)$. Then, by the
    observation above, $s_i^{\prime}=s_{i+1}^{\prime}=\cdots=s_{i+n-1}^{\prime}=0$ so that
    $s_i=s_{i+1}=\cdots=s_{i+n-1}=0$, contradicting the assumption that $S$ is an
    $\mathcal{OS}_q(n)$. 
\end{itemize}

\end{proof}

When $n=2$, this allows us to give the following.

\begin{corollary}
There exists an $\mathcal{SOS}_q(2)$ of period about one quarter of the maximum period defined by
\cite[Theorem 4.11]{Alhakim24a} for all $q \ge 5$.
 \end{corollary}

\begin{proof}
From \cite[Lemma 2.2]{Mitchell24} there exists an $\mathcal{OS}_q(2)$ with period either $q(q-1)/2$
($q$ odd) or $q(q-2)/2$ ($q$ even) for every $q\ge 3$. From Construction~\ref{construction:qprime},
this implies the existence of an $\mathcal{SOS}_{2q-1}(2)$ and an $\mathcal{SOS}_{2q}(2)$ with period either $q(q-1)/2$ ($q$ odd)
or $q(q-2)/2$ ($q$ even) for every $q\ge 3$. The result follows, since (by
Theorem~\ref{theorem:SOS_bound}) the maximum period for an $\mathcal{SOS}_{2q-1}(2)$ is
$(2q-2)^2/2$ and the maximum period for an $\mathcal{SOS}_{2q}(2)$ is $((2q-1)^2+1)/2$. 
\end{proof}

\subsection{A second construction}

We next modify the method given immediately above to double the period and so enable the
construction of special orientable sequences with period approximately half the maximum when $n=2$.
We do so by means of a general result regarding the relationship between a sequence and its
negative.

Following Alhakim et al.\ \cite{Alhakim24a} we make the following definition.

\begin{definition}  \label{definition:sidisjoint}
Suppose $S=(s_i)$ and $T=(t_i)$ are $n$-window sequences.  They are said to be
\emph{special-orientable-disjoint (s-disjoint)} if:
\begin{enumerate}
\item they are $n$-tuple disjoint, i.e.\ $\mathbf{s}_n(i) \neq \mathbf{t}_n(j)$ for any $i,j$;
\item they are orientable disjoint (o-disjoint), i.e.\ $\mathbf{s}_n(i) \neq \mathbf{t}_n(j)^R$
    for any $i,j$; and
\item they are negative orientable disjoint (n-disjoint), i.e.\ $\mathbf{s}_n(i) \neq
    -\mathbf{t}_n(j)^R$ for any $i,j$.
\end{enumerate}
\end{definition}

We can now state the following result.

\begin{theorem}  \label{theorem:negative-free}
Suppose $S$ is an $\mathcal{SOS}_q(n)$ with the property that, for any $n$-tuple $\mathbf{s}$, at
most one of $\mathbf{s}$ and $-\mathbf{s}$ is contained in $S$.  Then $S$ and $-S$ are s-disjoint.
\end{theorem}

\begin{proof}
$S$ and $-S$ are clearly $n$-tuple disjoint since we assumed that at most one of $\mathbf{s}$ and
$-\mathbf{s}$ is contained in $S$ for any $\mathbf{s}$.  Now $\mathbf{s}_n(i) \neq
-\mathbf{s}_n(j)^R$ for all $i,j$ since $S$ is an $\mathcal{NOS}_q(n)$, and hence $S$ and $-S$ are
o-disjoint. Finally, $\mathbf{s}_n(i) \neq -(-\mathbf{s}_n(j)^R) = \mathbf{s}_n(j)^R$ for all $i,j$
since $S$ is an $\mathcal{OS}_q(n)$, and hence $S$ and $-S$ are n-disjoint. 
\end{proof}

\begin{remark} \label{remark:merge}
It follows immediately from Theorem~\ref{theorem:negative-free} that if $S$ is an
$\mathcal{SOS}_q(n)$ of period $m$ with the property that, for any $n$-tuple $\mathbf{s}$, at most
one of $\mathbf{s}$ and $-\mathbf{s}$ is contained in $S$, and if in addition $S$ and $-S$ share an
$(n-1)$-tuple, then $S$ and $-S$ can be joined to form an $\mathcal{SOS}_q(n)$ with period $2m$.
\end{remark}

Next observe that any sequence $S$ obtained from  Construction~\ref{construction:qprime} has the property that, for any $n$-tuple $\mathbf{s}$, at most one of
$\mathbf{s}$ and $-\mathbf{s}$ is contained in $S$. This immediately motivates the following
construction.

\begin{construction}  \label{construction:special}
Let $q,q'$ be integers with $q' \ge 2q-1 >2$. For $x \in \mathbb{Z}_q$ we write $\underline x$ for
the non-negative integer in $\{0,1,\dots,q-1\}$ belonging to the residue class $x$ and $x'$ for the residue class of $\mathbb Z_{q'}$
that contains $\underline x$. Let $S=[s_0,s_1,\dots,s_{m-1}]$ be an $\mathcal{OS}_q(n)$. Let $S'=[s_0^{\prime},s_1^{\prime},\dots,s_{m-1}^{\prime}]$ be the sequence over
$\mathbb Z_{q'}$ obtained from $S$ in the obvious notational way. Let
$S''=[s_0'',s_1'',\dots,s_{2m-1}'']$ be the period sequence whose ring sequence is the
concatenation of the ring sequences of $S'$ and $-S'$.
\end{construction}

We next introduce some notation. Let $q,q'$ be integers with $q' \ge 2q-1 >2$. As in the
construction for $x \in \mathbb{Z}_q$, we write $\underline x$ for the integer in
$\{0,1,\dots,q-1\}$ belonging to the residue class $x$,  and $x'$ for the residue class of $\mathbb
Z_{q'}$ that contains $\underline x$. Similarly, for $y \in \mathbb Z_{q'}$ we write $\underline y$
for the integer in $\{0,1,\dots,q'-1\}$ belonging to the residue class $y$. Let $E_{q,q'} : \mathbf
Z_q \rightarrow \mathbf Z_{q'}$ be the mapping given by $E_{q,q'}(x)=x'$ for all $x \in \mathbf
Z_q$. Let $M_{q,q'} : \mathbf Z_{q'} \rightarrow \mathbf Z_{q}$ be the mapping given by
$M_{q,q'}(y)=x$ when $0 \le \underline y=\underline x \le q-1$  (so that  $x'=y$), $M(y)=0$ when $q
\le \underline y \le q'-q$, and $M_{q,q'}(y)=x$ when $q'-q+1 \le \underline y \le q'-1$ and
$\underline x = q'-\underline y$ (so that $x'=-y$). When $q$ and $q'$ are understood we simply
write $E$ and $M$ for $E_{q,q'}$ and $M_{q,q'}$ respectively.

Note that it follows immediately from the definitions of $E$ and $M$ that $M(-y)=M(y)$ for all $y
\in \mathbf Z_{q'}$, in particular $M(E(x))=M(-E(x))=x$ for all $x \in \mathbf Z_q$. We extend the
application of $E$ and $M$ to $n$-tuples and sequences in the natural way, that is by applying them
to each term. So, in Construction~\ref{corollary:special}, $S'=E(S)$ and $M(\mathbf
s_n''(i))=\mathbf s_n(i)$, $M(\mathbf s_n''^R(i))=\mathbf s_n^R(i)$ and $M(\mathbf
-s_n''(i))=\mathbf s_n(i)$ for all $i$.

\begin{theorem} \label{theorem:special}
If $S$ is an $\mathcal{OS}_q(n)$ and $S''$ is obtained from $S$ using
Construction~\ref{construction:special} then $S''$ is an $\mathcal{SOS}_{q'}(n)$ with $w_{q'}(S'')=0$.
\end{theorem}

\begin{proof}
We consider three cases.
\begin{itemize}
\item Suppose $\mathbf s_n''(i)=\mathbf s_n''(j)$.  Then $M(\mathbf s_n''(i))=M(\mathbf
    s_n''(j))$, that is $\mathbf s_n(i)=\mathbf s_n(j)$ and so $i\equiv j \pmod m$ as $S$ is an
    $n$-window sequence of period $m$. Since $\mathbf s_n''(i+m)=-\mathbf s_n''(i)$ and
    $\mathbf s_n(i)$ cannot have  every term equal to $0$, we deduce that $i\equiv j \pmod
    {2m}$.
\item Suppose $\mathbf s_n''(i)=\mathbf s_n''^R(j)$. Then $M(\mathbf s_n''(i))=M(\mathbf
    s_n''^R(j))$, that is $\mathbf s_n(i)=\mathbf s_n^{R}(j)$ which is impossible as $S$ is an
    orientable sequence.
\item Finally, suppose $\mathbf s_n''(i)=-\mathbf s_n''^R(j)$. Then $M(\mathbf
    s_n''(i))=M(-\mathbf s_n''^R(j))$, that is $\mathbf s_n(i)=\mathbf s_n^{R}(j)$ which is
    impossible as $S$ is an orientable sequence.

\end{itemize}
The result follows, observing that $w_{q'}(S'')= w_{q'}(S')+w_{q'}(-S')= w_{q'}(S')-w_{q'}(S')= 0$.
\end{proof}

The following simple example demonstrates Construction~\ref{construction:special}.

\begin{example}  \label{example:S_and_-S}
First observe that $S=[01234~02413]$ is an $\mathcal{OS}_5(2)$ (obtained using Construction 5.3 of
\cite{Alhakim24})\footnote{Here and in other examples the spaces are included simply to make
reading easier.}.

If we put $q'=9$, then $S'=[01234~02413]$ and
\[ S''=S'||(-S')=[01234~02413~08765~07586] \]
(where $||$ denotes sequence concatenation).  It follows from Theorem~\ref{theorem:special} that
$S''$ is an $\mathcal{SOS}_9(2)$.

We can also perform the same construction with $q'=10$.  In this case
\[ S''=S'||(-S')=[01234~02413~09876~08697] \]
and $S''$ is an $\mathcal{SOS}_{10}(2)$.
\end{example}

\begin{corollary} \label{corollary:special}
There exists an $\mathcal{SOS}_{q}(2)$, of period
\begin{align*}
\frac{q(q-4)}{4}      &  ~~\mbox{if}~~q\equiv 0 \pmod 4,\\
\frac{(q+1)(q-1)}{4}  &  ~~\mbox{if}~~q\equiv 1 \pmod 4,\\
\frac{q(q-2)}{4}      &  ~~\mbox{if}~~q\equiv 2 \pmod 4,\\
\frac{(q+1)(q-3)}{4}  &  ~~\mbox{if}~~q\equiv 3 \pmod 4,
\end{align*}
for all $q \ge 5$.
\end{corollary}

\begin{proof}
Suppose $q \ge 5$.  If $q\equiv 0 \pmod 4$ then $\frac{q}{2}$ is even and by \cite[Lemma
2.2]{Mitchell24} there exists an $\mathcal{OS}_{\frac{q}{2}}(2)$ with period
$\frac{q}{2}(\frac{q}{2}-2)/2$; hence by Theorem~\ref{theorem:special} there exists an $SOS_q(2)$
of period $\frac{q(q-4)}{4}$.

If $q\equiv 1 \pmod 4$ then $\frac{q+1}{2}$ is odd and by \cite[Lemma 2.2]{Mitchell24} there exists
an $OS_{\frac{q+1}{2}}(2)$ with period $\frac{q+1}{2}(\frac{q+1}{2}-1)/2$; hence by
Theorem~\ref{theorem:special} there exists an $SOS_q(2)$ of period $\frac{(q+1)(q-1)}{4}$.

If $q\equiv 2 \pmod 4$ then $\frac{q}{2}$ is odd and by \cite[Lemma 2.2]{Mitchell24} there exists
an $\mathcal{OS}_{\frac{q}{2}}(2)$ with period $\frac{q}{2}(\frac{q}{2}-1)/2$; hence by
Theorem~\ref{theorem:special} there exists an $SOS_q(2)$ of period $\frac{q(q-2)}{4}$.

If $q\equiv 3 \pmod 4$ then $\frac{q+1}{2}$ is even and by \cite[Lemma 2.2]{Mitchell24} there exists
an $OS_{\frac{q+1}{2}}(2)$ with period $\frac{q+1}{2}(\frac{q+1}{2}-2)/2$; hence by
Theorem~\ref{theorem:special} there exists an $SOS_q(2)$ of period $\frac{(q+1)(q-3)}{4}$.

\end{proof}

We remark that the period of these Special Orientable Sequences of order $2$ is approximately half
that of the maximum period given by Theorem~\ref{theorem:SOS_bound}.

\subsection{Extending the construction}  \label{subsection:extending_the_construction}

We now further modify the previous constructions, doubling the period again, to enable us to obtain special orientable
sequences which have period of the same order as the bound of Theorem~\ref{theorem:SOS_bound} when
$n=2$.

\begin{construction} \label{construction:sequenceT}
Let $q,q'$ be integers with $q' \ge 2q >3$. Let $S=[s_0,\dots,s_{m-1}]$ be an $\mathcal{OS}_q(n)$
and let $S'=[s_0^{\prime},\dots,s_{m-1}^{\prime}]$ be the $\mathcal{OS}_{q'}(n)$ constructed as in
Construction~\ref{construction:special}.  Let $T=[t_0,t_1,\dots,t_{m-1}]$ be the sequence over
$\mathbb{Z}_{q'}$ such that $t_i=(-1)^{i+m-1}s_i^{\prime}$ for $i=0,\dots,m-1$ unless
$s_i^{\prime}=0$ in which case $t_i=(-1)^{i+m-1}q$.
\end{construction}

\begin{lemma} \label{SOST}
The sequence $T$ of Construction~\ref{construction:sequenceT} is an $\mathcal{SOS}_n(q')$.
\end{lemma}

\begin{proof}
We consider three cases.
\begin{itemize}
\item Suppose $\mathbf t_n(i)=\mathbf t_n(j)$.  Then $M(\mathbf t_n(i))=M(\mathbf t_n(j))$, so
    that $\mathbf s_n(i)=\mathbf s_n(j)$ since $M(q)=M(-q)=0$.  Hence $i=j \pmod m$ as $S$ is
    an $n$-window sequence of period $m$.
\item Suppose $\mathbf t_n(i)=\mathbf t_n^{R}(j)$. Then $M(\mathbf t_n(i))=M(\mathbf
    t_n^{R}(j))$, so that  $\mathbf s_n(i)=\mathbf s_n^{R}(j)$, which is impossible as $S$ is
    an orientable sequence.
\item Finally, suppose $\mathbf t_n(i)=-\mathbf t_n^{R}(j)$. Then $M(\mathbf t_n(i))=M(-\mathbf
    t_n^{R}(j))$, that is $\mathbf s_n(i)=\mathbf s_n^{R}(j)$, which is impossible as $S$ is an
    orientable sequence.

\end{itemize}
The result follows.
\end{proof}

An example of Construction~\ref{construction:sequenceT} follows.

\begin{example}  \label{example:T}
As in Example~\ref{example:S_and_-S}, let $S=[01234~02413]$ be the $\mathcal{OS}_5(2)$ obtained
using Construction 5.3 of \cite{Alhakim24}.

If we put $q'=10$, then, as $m=10$, $(-1)^{m-1}=-1$ and
\[ T=[51836~58493]. \]
It follows from Lemma~\ref{SOST} that $T$ is an $\mathcal{SOS}_{10}(2)$.
\end{example}

We next show that we can adjoin $-T$ to $T$ to obtain an $\mathcal{SOS}$ with twice the period,
just as was the case with $S'$.

\begin{construction} \label{construction:sequenceTprime}
Let $q,q',n$ be integers with $q' \ge 2q >3$ and $n>1$.  Let $S=[s_0,\dots,s_{m-1}]$ be an
$\mathcal{OS}_q(n)$. Let $T=[t_0,t_1,\dots,t_{m-1}]$ be as in
Construction~\ref{construction:sequenceT}. Let $T'$ be the sequence whose ring sequence is the
concatenation of the ring sequences of $T$ and $-T$.
\end{construction}

\begin{theorem} \label{SOSTprime}
If $S$ is an $\mathcal{OS}_q(n)$ and $T'$ is obtained from $S$ using
Construction~\ref{construction:sequenceTprime} then $T'$ is an $\mathcal{SOS}_{q'}(n)$ with
$w_{q'}(T')=0$.
\end{theorem}

\begin{proof}
We consider three cases.
\begin{itemize}
\item Suppose $\mathbf t_n^{\prime}(i)=\mathbf t_n^{\prime}(j)$.  Then $M(\mathbf
    t_n^{\prime}(i))=M(\mathbf t_n^{\prime}(j))$, so that $\mathbf s_n(i)=\mathbf s_n(j)$ since
    $M(q)=M(-q)=0$.  Hence $i\equiv j \pmod m$ as $S$ is an $n$-window sequence of period $m$.
    Since $\mathbf t_n^{\prime}(i+m) = -\mathbf t_n^{\prime}(i) \not= \mathbf t_n^{\prime}(i)$,
    as $\mathbf s_n(i)$ is not the all $0$ tuple, we must have $j=i \pmod {2m}$.
\item Suppose $\mathbf t_n^{\prime}(i)=\mathbf t_n^{'R}(j)$. Then $M(\mathbf
    t_n^{\prime}(i))=M(\mathbf t_n^{'R}(j))$, so that  $\mathbf s_n(i)=\mathbf s_n^{R}(j)$,
    which is impossible as $S$ is an orientable sequence.
\item Finally, suppose $\mathbf t_n^{\prime}(i)=-\mathbf t_n^{'R}(j)$. Then $M(\mathbf
    t_n^{\prime}(i))=M(-\mathbf t_n^{'R}(j))$, that is $\mathbf s_n(i)=\mathbf s_n^{R}(j)$,
    which is impossible as $S$ is an orientable sequence.

\end{itemize}
The result follows, observing that $w_{q'}(T')= w_{q'}(T)+w_{q'}(-T)= w_{q'}(T)-w_{q'}(T)= 0$.
\end{proof}

We extend our previous example to give an example of
Construction~\ref{construction:sequenceTprime}.

\begin{example}  \label{example:T_and_-T}
As in Examples~\ref{example:S_and_-S} and \ref{example:T}, let $S=[01234~02413]$ be the
$\mathcal{OS}_5(2)$ obtained using Construction 5.3 of \cite{Alhakim24}.

If we put $q'=10$, then, as in Example~\ref{example:T}
\[ T=[51836~58493]. \]
We then have that
\[ T' = T||(-T) = [51836~58493~59274~52617]. \]
It follows from Theorem~\ref{SOSTprime} that $T'$ is an $\mathcal{SOS}_{10}(2)$.
\end{example}

We complete the extended construction by combining the sequence $T'$ generated using
Construction~\ref{construction:sequenceTprime} with the sequence $S''$ generated using
Construction~\ref{construction:special}.

\begin{corollary}  \label{corollary:STsdisjoint}
Let $q,q'$ be integers with $q' \ge 2q+1 >4$.  Let $S=[s_0,\dots,s_{m-1}]$ be an
$\mathcal{OS}_q(n)$ of period $m$ with $s_0=0$. Let $S^{\prime\prime}$ be obtained from $S$ as in
Construction~\ref{construction:special} and let $T'$ be obtained from $S$ as in
Construction~\ref{construction:sequenceTprime}. Then $S^{''}$ and $T'$ are s-disjoint.
\end{corollary}

\begin{proof}
We consider three cases.
\begin{itemize}
\item Suppose $\mathbf s_n^{\prime\prime}(i)=\mathbf t_n^{\prime}(j)$.  Then $M(\mathbf
    s_n^{\prime\prime}(i))=M(\mathbf t_n^{\prime}(j))$, so that $\mathbf s_n(i)=\mathbf s_n(j)$
    since $M(q)=M(-q)=0$. Hence $i\equiv j \pmod m$ as $S$ is an $n$-window sequence of period
    $m$. Suppose $0 \le i \le 2m-1$. Then for some $k$ with $0 \le k \le n-1$, we have that
    $s_{i+\ell}^{\prime\prime}$ for $\ell=0,\dots,k$ all lie in $\{0,1,\dots,q-1\}$  or in
    $\{0,-1,\dots,-(q-1)\}$ and $s_{i+\ell}^{\prime\prime}$ for $\ell=k+1,\dots,{n-1}$ all lie
    in $\{0,-1,\dots,-(q-1)\}$  or in $\{0,1,\dots,(q-1)\}$ respectively while the terms of
    $\mathbf t_n^{\prime}(j)$ alternate between the two sets $\{1,\dots,q\}$ and
    $\{-1,\dots,-q\}$ unless $m$ is even and $\mathbf t^{\prime}_n(j)$ contains $t'_{\alpha
    m-1}$ and $t'_{\alpha m}$ for some $\alpha \ge 1$. Since $n>1$ and
    $s^{\prime\prime}_{\alpha m}=0$ so that $t^{\prime}_{\alpha m}=\pm q$ this is not possible.
    It follows that $S^{''}$ and $T'$ are $n$-window disjoint.
\item Suppose $\mathbf s''_n(i)=\mathbf t_n^{\prime R}(j)$. Then $M(\mathbf
    s_n^{\prime\prime}(i))=M(\mathbf t_n^{\prime R}(j))$, so that $\mathbf s_n(i)=\mathbf
    s_n^R(j)$ since $M(q)=M(-q)=0$. Hence $i\equiv j \pmod m$ as $S$ is an orientable sequence
    of period $m$. Now a similar argument as above about where the terms of $\mathbf
    s^{\prime\prime}_n(i)$ and $\mathbf t_n^{\prime R}(j)$ lie shows that the supposition is
    impossible and it follows that $S''$ and $T'$ are $o$-disjoint.
\item Finally, suppose $\mathbf s_n^{\prime\prime}(i)=-\mathbf t_n^{\prime R}(j)$. Then
    $M(\mathbf s_n^{\prime\prime}(i))=M(-\mathbf t_n^{\prime R}(j))$, so that $\mathbf
    s_n(i)=\mathbf s_n^R(j)$ since $M(q)=M(-q)=0$. Hence $i\equiv j \pmod m$ as $S$ is an
    orientable sequence of period $m$. Now a similar argument as above about where the terms of
    $\mathbf s^{\prime\prime}_n(i)$ and $\mathbf t_n^{\prime R}(j)$ lie shows that the
    supposition is impossible and it follows that $S''$ and $T'$ are $n$-disjoint.
\end{itemize}
The result follows.
\end{proof}

\begin{corollary} \label{corollary:specialconcatenate}
Suppose $q,q',n$ are integers satisfying $q'\geq 2q+1$, $q>1$ and $n>1$.  If $S$ is an
$\mathcal{OS}_q(n)$ of period $m$ with $s_0=0$, and $S''$ and $T'$ are obtained from $S$ using
Constructions~\ref{construction:special} and~\ref{construction:sequenceTprime}, then the ring
sequences of $S''$ and $T'$ may be concatenated to obtain the ring sequence of an
$\mathcal{SOS}_{q'}(n)$ $U$ of period $4m$, where $w_{q'}(U)=0$.
\end{corollary}

\begin{proof}
As $S^{''}$ and $T'$ are s-disjoint $\mathcal{SOS}_{q'}(n)$  we need only check that the $n$-tuples
$\mathbf u_n(i)$, $i=2m-n+1,\dots, 2m-1$ and $i=4m-n+1,\dots, 4m-1$ do not appear as $\mathbf
u_n(j)$ for any $j \not\equiv i \pmod{4m}$, nor as $\mathbf u_n^R(j)$ for any $j$, nor as $-\mathbf
u_n^R(j)$ for any $j$. Suppose $\mathbf u_n(i)$, with $i \in \{2m-n+1,\dots, 2m-1\}$ or with $i \in
\{4m-n+1,\dots, 4m-1\}$ equals $\mathbf u_n(j)$ for some $j$. Then $M(\mathbf u_n(i))=M(\mathbf
u_n(j))$ are $n$-tuples of $S$ so that $j \equiv i \pmod m$. We now need only check that for $\ell
= 1,\dots n-1$ the four $n$-tuples $\mathbf u_{m-\ell},\mathbf u_{2m-\ell},\mathbf
u_{3m-\ell},\mathbf u_{4m-\ell}$ are distinct. This follows if the four $2$-tuples
$(u_{m-1},u_{m})$, $(u_{2m-1},u_{2m})$, $(u_{3m-1},u_{3m})$, $(u_{4m-1},u_{4m})$ are distinct. That
is  $(s^{\prime}_{m-1},0)$, $(-s^{\prime}_{m-1},(-1)^{m-1}q)$, $(t_{m-1},(-1)^mq)$, $(-t_{m-1},0)$
are distinct.  This is easily checked, knowing that $t_{m-1} = (-1)^{2m-2}s^{\prime}_{m-1} \neq
-s^{\prime}_{m-1}$ unless $s^{\prime}_{m-1}=0$ in which case $t_{m-1}=(-1)^{2m-2}q$. It follows
that $j=i \pmod{4m}$ and $U$ is an $n$-window sequence.

Similar arguments as before, using the mapping $M$, show that $\mathbf u_n(i)$ does not equal
$\mathbf u_n^R(j)$ or $-\mathbf u_n^R(j)$ for any $j$, so $U$ is both orientable and negative
orientable. Thus $U$ is an $\mathcal {SOS}_{q'}(n)$. The result follows, observing that $w_{q'}(U)
= w_{q'}(S^{''})+w_{q'}(-S^{''})+w_{q'}(T')+w_{q'}(-T') =
w_{q'}(S^{''})-w_{q'}(-S^{''})+w_{q'}(T')-w_{q'}(-T') = 0$.
\end{proof}

A simple example of Corollary~\ref{corollary:specialconcatenate} is as follows.

\begin{example}  \label{example:SOS_11_2}
Suppose $q=5$, $q'=11$ and $n=2$. As previously, we build upon the $\mathcal{OS}_5(2)$ with ring
sequence $S=[01234~02413]$.  Analogously to the second part of Example~\ref{example:S_and_-S} we
have
\[ S'' = [0,1,2,3,4,~0,2,4,1,3,~0,10,9,8,7,~0,9,7,10,8]. \]
Analogously to Example~\ref{example:T_and_-T} we have
\[ T' = [6,1,9,3,7,~5,9,4,10,3,~5,10,2,8,4,~6,2,7,1,8]. \]
We simply concatenate them to obtain
\begin{align*}
U = & [0,1,2,3,4, ~0,2,4,1,3,~0,10,9,8,7,~0,9,7,10,8,\\
    & ~6,1,9,3,7,~5,9,4,10,3,~5,10,2,8,4,~6,2,7,1,8]
\end{align*}
which by Corollary~\ref{corollary:specialconcatenate} is an $\mathcal{SOS}_{11}(2)$.
\end{example}

\begin{corollary}  \label{corollary:SOSq2}
There exists an $\mathcal{SOS}_{q}(2)$ of period:
\begin{align*}
\frac{(q-2)(q-4)}{2} &  ~~\mbox{if}~~q\equiv 0 \pmod 4, \\
\frac{(q-1)(q-5)}{2} &  ~~\mbox{if}~~q\equiv 1 \pmod 4, \\
\frac{(q-2)(q-6)}{2} &  ~~\mbox{if}~~q\equiv 2 \pmod 4, \\
\frac{(q-1)(q-3)}{2} &  ~~\mbox{if}~~q\equiv 3 \pmod 4,
\end{align*}
for all $q \ge 6$.
\end{corollary}

\begin{proof}
Suppose $q \ge 6$.
\begin{itemize}
\item If $q\equiv 0 \pmod 4$ then $\frac{q-2}{2}$ is odd, and by \cite[Lemma 2.2]{Mitchell24}
    there exists an $\mathcal{OS}_{\frac{q-2}{2}}(2)$ with period
    $\frac{q-2}{2}(\frac{q-2}{2}-1)/2$. So by Corollary~\ref{corollary:specialconcatenate}
    there exists an $\mathcal{SOS}_q(2)$ of period
    $4\frac{q-2}{2}(\frac{q-2}{2}-1)/2=\frac{(q-2)(q-4)}{2}$.

\item If $q\equiv 1 \pmod 4$ then $\frac{q-1}{2}$ is even, and by \cite[Lemma 2.2]{Mitchell24}
    there exists an $\mathcal{OS}_{\frac{q-1}{2}}(2)$ with period
    $\frac{q-1}{2}(\frac{q-1}{2}-2)/2$. So by Corollary~\ref{corollary:specialconcatenate}
    there exists an $\mathcal{SOS}_q(2)$ of period
    $4\frac{q-1}{2}(\frac{q-1}{2}-2)/2=\frac{(q-1)(q-5)}{2}$.

\item If $q\equiv 2 \pmod 4$ then $\frac{q-2}{2}$ is even, and by \cite[Lemma 2.2]{Mitchell24}
    there exists an $\mathcal{OS}_{\frac{q-2}{2}}(2)$ with period
    $\frac{q-2}{2}(\frac{q-2}{2}-2)/2$. So by Corollary~\ref{corollary:specialconcatenate}
    there exists an $\mathcal{SOS}_q(2)$ of period
    $4\frac{q-2}{2}(\frac{q-2}{2}-2)/2=\frac{(q-2)(q-6)}{2}$.

\item If $q\equiv 3 \pmod 4$ then $\frac{q-1}{2}$ is odd, and by \cite[Lemma 2.2]{Mitchell24}
    there exists an $\mathcal{OS}_{\frac{q-1}{2}}(2)$ with period
    $\frac{q-1}{2}(\frac{q-1}{2}-1)/2$. So by Corollary~\ref{corollary:specialconcatenate}
    there exists an $\mathcal{SOS}_q(2)$ of period
    $4\frac{q-1}{2}(\frac{q-1}{2}-1)/2=\frac{(q-1)(q-3)}{2}$.
\end{itemize}
\end{proof}

Observe that all the constructed sequences have $q'$-ary weight zero.

\subsection{Adjusting the weight}  \label{subsection:adjusting_the_weight}

Our main objective in giving the above constructions is to provide `starter sequences' for certain
constructions of Alhakim et al.\ \cite{Alhakim24}.  However, all the sequences constructed here
have weight zero; in particular, the sequence $U$ obtained in
Corollary~\ref{corollary:specialconcatenate} satisfies $w_{q'}(U)=0$.  We would ideally like to
construct sequences $U^*$ such that $w_{q'}(U^*)$ is coprime to $q'$.  Therefore we next describe
how to modify the sequences $U$ of Corollary~\ref{corollary:specialconcatenate} in the case $n=2$
to obtain sequences with precisely this property.

We first need the following simple result.

\begin{lemma}  \label{lemma:special_OSq2}
Suppose $q>4$.  Then, for any distinct $x,y,z$ in $\mathbb{Z}_q$, there exists an
$\mathcal{OS}_q(2)$ of maximal period, i.e.\ of period $q(q-1)/2$ ($q$ odd) or $q(q-2)/2$ ($q$
even), such that its ring sequence has the form $[xyzx\ldots]$. Moreover, if $x,y,z\not=0$ then there exists an $\mathcal{OS}_q(2)$ of maximal period such that its ring sequence has the form $[0xyzx\ldots]$.
\end{lemma}

\begin{proof}
If $q$ is odd then, from Lemma 2.2 of \cite{Mitchell24}, there exists an $\mathcal{OS}_q(2)$ of
period $q(q-1)/2$ corresponding to an Eulerian circuit in $K_q$, the complete graph on $q$
vertices.  Every vertex has degree $q-1$, which is at least 4 since $q\geq 5$, and hence there
exists an Eulerian circuit in $K_q$ starting with the edges $(x,y)$, $(y,z)$, $(z,x)$, and, should
$x,y,z\not=0$, an Eulerian circuit starting with the edges $(0,x),(x,y)$, $(y,z)$, $(z,x)$. The
result follows.

If $q$ is even, then (again from Lemma 2.2 of \cite{Mitchell24}), there exists an
$\mathcal{OS}_q(2)$ of period $q(q-2)/2$ corresponding to an Eulerian circuit in $K^*_q$, where
$K^*_q$ is $K_q$ with an arbitrary one-factor removed.  Since $q\geq 6$, it is simple to choose a
one-factor which avoids the edges $(x,y)$, $(y,z)$, and $(z,x)$ or, should $x,y,z\not=0$, the edges
$(0,x)$, $(x,y)$, $(y,z)$, and $(z,x)$; moreover the vertices in $K^*_q$ will have degree at least
4. As a result there will exist an Eulerian circuit in $K^*_q$ starting with the edges $(x,y)$,
$(y,z)$, $(z,x)$ and, should $x,y,z\not=0$, with edges $(0,x)$, $(x,y)$, $(y,z)$, and $(z,x)$. The
result follows.
\end{proof}

\begin{construction}  \label{construction:specialconcatenate_truncated}
Suppose $q>2$ and $q'=2q+1$ or $q'=2q+2$. If $q\geq 5$ and $q'=2q+1$, set $x=2$, $y=q-2$ and
$z=q-1$ (and so $x+y+z=2q-1$). If $q\geq 7$ and $q'=2q+2$, set $x=4$, $y=q-2$ and $z=q-1$ (and so
$x+y+z=2q+1$). Otherwise set $x$, $y$ and $z$ according to Table~\ref{table:xyz_choices}.

\begin{table}[htb]
\caption{Choosing $x$, $y$ and $z$} \label{table:xyz_choices}
\begin{center}
\begin{tabular}{c|c||c|c|c||c}
       $q$ & $q'$ & $x$ & $y$ & $z$ & $x+y+z$ \\
\hline\hline 5 & 12 & 0 & 1 & 4 & 5 \\
\hline 6 & 14 & 0 & 1 & 2 & 3 \\
\hline
\end{tabular}
\end{center}
\end{table}

First observe that, in all cases $x$, $y$ and $z$ are distinct and $y,z\not=0$.  By inspection it
also holds that $x+y+z$ is coprime to $q'$ for all possible choices of $q$ and $q'$.

Suppose $S$ is an $\mathcal{OS}_q(2)$ of maximal period $m$ (i.e.\ of period $q(q-1)/2$ ($q$ odd)
or $q(q-2)/2$ ($q$ even)), such that its ring sequence has the form $[xyzx\ldots]$ or, should
$x,y,z\not=0$, the form $[0xyzx\ldots]$, which exists from Lemma~\ref{lemma:special_OSq2}.
Construct $U$ from $S$ using the method of Corollary~\ref{corollary:specialconcatenate}.  Observe
that, from the method of construction, the ring sequence for $U$ has the form $[xyzx\ldots]$ or,
should $x\not=0$, the form $[0xyzx\ldots]$.  Finally, construct $U^*$ from $U$ by deleting the
cycle $[xyz]$ from its ring sequence.
\end{construction}

\begin{theorem} \label{theorem:specialconcatenate_truncated}
Suppose $q>4$ and $q'=2q+1$ or $q'=2q+2$.  If $U^*$ is constructed according to the method of
Construction~\ref{construction:specialconcatenate_truncated} then it is an $\mathcal{SOS}_{q'}(2)$
of period $2q(q-1)-3$ ($q$ odd) or $2q(q-2)-3$ ($q$ even) where in every case $w_{q'}(U^*)$ is
coprime to $q'$.
\end{theorem}

\begin{proof}
By Corollary~\ref{corollary:specialconcatenate}, the sequence $U$ is an $\mathcal{SOS}_{q'}(2)$ of
period $2q(q-1)$ ($q$ odd) or $2q(q-1)$ ($q$ odd) where $w_{q'}(U)=0$.  The result now follows
immediately by observing that constructing $U^*$ from $U$ does not add any new 2-tuples, that
$w_{q'}(U^*)=q'-(x+y+z)$, and, as noted above, $x+y+z$ is coprime to $q'$.
\end{proof}

The following brief example shows the operation of this construction.

\begin{example}  \label{example:weight_adjusted_SOS_11_2}
Suppose $q=5$ and $q'=11$.  In this case $x=2$, $y=3$ and $z=4$, and so we need an
$\mathcal{OS}_2(q)$ of maximal period with ring sequence of the form $[02342\ldots]$. An example of
such a sequence is $S=[02342~10314]$. Then $S=[02342~10314]$ and
\[ S''=S'||(-S')=[0,2,3,4,2,~1,0,3,1,4,~0,9,8,7,9,~10,0,8,10,7] \]
(where $||$ denotes sequence concatenation).  It follows from Theorem~\ref{theorem:special} that
$S''$ is an $\mathcal{SOS}_{11}(2)$. We next have
\[ T=[6,2,8,4,9,~1,6,3,10,4], \]
where, from Lemma~\ref{SOST}, $T$ is an $\mathcal{SOS}_{11}(2)$. Then
\[ T' = T||(-T) = [6,2,8,4,9,~1,6,3,10,4,~5,9,3,7,2,~10,5,8,1,7]. \]
We next concatenate $S''$ and $T'$ to obtain
\begin{align*}
U = & [0,2,3,4,2,~1,0,3,1,4,~0,9,8,7,9,~10,0,8,10,7 \\
    & ~6,2,8,4,9,~1,6,3,10,4,~5,9,3,7,2,~10,5,8,1,7].
\end{align*}
which by Corollary~\ref{corollary:specialconcatenate} is an $\mathcal{SOS}_{11}(2)$. Finally we
simply delete the cycle $[234]$ from $U$ to obtain
\begin{align*}
U^* = & [0,2,~1,0,3,1,4,~0,9,8,7,9,~10,0,8,10,7 \\
    & ~6,2,8,4,9,~1,6,3,10,4,~5,9,3,7,2,~10,5,8,1,7].
\end{align*}
which is an $\mathcal{SOS}_{11}(2)$ of period 37 with $w_{11}(U^*)=2$.
\end{example}

We also have the following simple corollary, which follows immediately from
Corollary~\ref{corollary:SOSq2}.

\begin{corollary}    \label{corollary:SOSq2_truncated}
There exists an $\mathcal{SOS}_{q}(2)$ $U^*$ of period:
\begin{align*}
\frac{(q-2)(q-4)}{2}-3     &  ~~\mbox{if}~~q\equiv 0 \pmod 4, \\
\frac{(q-1)(q-5)}{2}-3 &  ~~\mbox{if}~~q\equiv 1 \pmod 4, \\
\frac{(q-2)(q-6)}{2}-3     &  ~~\mbox{if}~~q\equiv 2 \pmod 4, \\
\frac{(q-1)(q-3)}{2}-3 &  ~~\mbox{if}~~q\equiv 3 \pmod 4,
\end{align*}
for all $q \ge 11$, where $w_{q}(U^*)$ is coprime to $q$.
\end{corollary}


\section{Good special orientable sequences}  \label{section:good_SOS}

We next consider how to construct \emph{good} special orientable sequences, given that this
additional property enables us to apply certain recursive constructions from Alhakim at al.
\cite{Alhakim24a}. We first need the following.

\begin{definition}[\cite{Alhakim24a}]
An orientable (respectively negative orientable) sequence with the property that any run of $0$ has
length at most $n-2$ is said to be \emph{good}.
\end{definition}

\subsection{An initial observation}

We immediately have the following, although the sequences have period only of the order of half the
bound of Theorem~\ref{theorem:SOS_bound}.

\begin{theorem}
There exists a good $\mathcal{SOS}_{q}(2)$, of period
\begin{align*}
\frac{q(q-4)}{4}     &  ~~\mbox{if}~~q\equiv 0 \pmod 4,\\
\frac{(q+1)(q-1)}{4} &  ~~\mbox{if}~~q\equiv 1 \pmod 4,\\
\frac{q(q-2)}{4}     &  ~~\mbox{if}~~q\equiv 2 \pmod 4,\\
\frac{(q+1)(q-3)}{4} &  ~~\mbox{if}~~q\equiv 3 \pmod 4,
\end{align*}
for all $q \ge 5$.
\end{theorem}

\begin{proof}
The $OS_{q'}(n)$ $T'$ of Theorem~\ref{SOSTprime} is good by construction. The result
follows using the same argument as in Corollary~\ref{corollary:special}. 
\end{proof}

In the remainder of this section we show how we can do considerably better than this.

\subsection{A simple modification}

A simple method of constructing a good special orientable sequence arises from the observation that
an $\mathcal{SOS}_{q}(n)$ that possesses no zeros is automatically a good $\mathcal{SOS}_{q}(n)$.
With this is mind we modify the sequences $U$ of Corollary~\ref{corollary:specialconcatenate}. Note
that such a sequence $U$ will always contain an even number of zeros, since in the sequences $S'$
and $-S'$ that are concatenated to construct $U$, every zero in $S'$ will give rise to a zero in
$-S'$.

\begin{construction}  \label{construction:goodify_specialconcatenate}
Suppose $q,q',n$ are integers satisfying $q'\geq2q+2$, $q>1$ and $n>1$.  Suppose $U$ is an
$\mathcal{SOS}_{q'}(n)$ constructed according to Corollary~\ref{corollary:specialconcatenate}. Then
let $U'$ be derived from $U$ by replacing half of the zeros with $q+1$ and the other half with
$q'-q-1$.
\end{construction}

\begin{theorem}  \label{theorem:goodify_specialconcatenate}
Suppose $q,q',n$ are integers satisfying $q'\geq2q+2$, $q>1$ and $n>1$.  Suppose $U$ is an
$\mathcal{SOS}_{q'}(n)$ constructed according to Corollary~\ref{corollary:specialconcatenate}. If
$U'$ is derived from $U$ using Construction~\ref{construction:goodify_specialconcatenate}, then
$U'$ is a good $\mathcal{SOS}_{q'}(n)$ of the same period as $U$, and $w_{q'}(U')=0$.
\end{theorem}

\begin{proof}
If we can show that $U$ does not contain any occurrences of $q+1$ or $q'-q-1$ then the main result
will follow immediately.  Now $U$ is constructed by concatenating sequences $S''$ and $T'$,
obtained using Constructions~\ref{construction:special} and~\ref{construction:sequenceTprime}, so
we next examine these two sequences.

$S''$ is obtained by concatenating sequences $S'$ and $-S'$, where $S$ is an $\mathcal{OS}_{q}(n)$.
Now $S'$ contains only elements between 0 and $q-1$ inclusive, and $-S'$ contains only $0$ or
elements between $q'-q+1$ and $q'-1$.  Since $q'\geq2q+2$, $q'-q+1\geq q+2$.  Hence $S''$ does not
contain any occurrences of $q+1$ or $q'-q-1$.

$T'$ is constructed as the concatenation of sequences $T$ and $-T$, where an element of $T$ is in
one of the ranges $[1,q]$ and $[q'-q,q'-1]$. Also, as before, since $q'\geq2q+2$ we have $q'-q\geq
q+2$. Hence $T$ does not contain any occurrences of $q+1$ or $q'-q-1$.  Now consider $-T$.  It
follows immediately that the elements of $-T$ are in the same ranges as $T$.  Hence $T'$ will not
contain any instances of $q+1$ or $q'-q-1$.

It remains to show that $w_{q'}(U')=0$.  From Corollary~\ref{corollary:specialconcatenate} we know
that $w_{q'}(U)=0$.  The only changes made to $U$ are to add $q+1$ to half of the zeros and
$q'-q-1$ to the other half.  Thus, if $U$ contains $2s$ zeros, $w(U')\equiv w(U)\equiv
0+s(q+1+q'-q-1)\equiv 0\pmod{q'}$, and the result follows.
\end{proof}

\begin{remark}
A good $\mathcal{SOS}_{q'}(n)$ with identical parameters could be constructed by taking an
$\mathcal{SOS}_{q'-1}(n)$ constructed according to Corollary~\ref{corollary:specialconcatenate},
and `adding one' to every element.  More formally, since each element of $U$ is in
$\mathbb{Z}_{q'-1}$, we treat every element of $U$ as an integer, add one, and then treat the
result as an element of $\mathbb{Z}_{q'}$.
\end{remark}

The following example is similar to Examples~\ref{example:S_and_-S} and~\ref{example:SOS_11_2}.

\begin{example}  \label{example:good_SOS_12_2}
Suppose $q=5$, $q'=12$ and $n=2$. As previously, we build upon the $\mathcal{OS}_5(2)$ with ring
sequence $S=[01234~02413]$.  Analogously to the second part of Example~\ref{example:S_and_-S} we
have
\[ S'' = [0,1,2,3,4,~0,2,4,1,3,~0,11,10,9,8,~0,10,8,11,9]. \]
Analogously to Example~\ref{example:T_and_-T} we have
\[ T' = [7,1,10,3,8,~5,10,4,11,3,~5,11,2,9,4,~7,2,8,1,9]. \]
We simply concatenate them to obtain
\begin{align*}
U = & [0,1,2,3,4, ~0,2,4,1,3,~0,11,10,9,8,~0,10,8,11,9,\\
    & ~7,1,10,3,8,~5,10,4,11,3,~5,11,2,9,4,~7,2,8,1,9]
\end{align*}
which by Corollary~\ref{corollary:specialconcatenate} is an $\mathcal{SOS}_{12}(2)$ (and we can
observe it contains no occurrences of $q+1=6$).

Finally, we replace every 0 with $q+1=6$ (since in this case $q'-q-1=q+1$) to obtain
\begin{align*}
U' = & [6,1,2,3,4, ~6,2,4,1,3,~6,11,10,9,8,~6,10,8,11,9,\\
    & ~7,1,10,3,8,~5,10,4,11,3,~5,11,2,9,4,~7,2,8,1,9]
\end{align*}
which by Theorem~\ref{theorem:goodify_specialconcatenate} is a good $\mathcal{SOS}_{12}(2)$ with
$w_{12}(U')=0$.
\end{example}

\begin{corollary}  \label{corollary:goodSOSq2}
There exists a good $\mathcal{SOS}_{q}(2)$ of period:
\begin{align*}
\frac{(q-2)(q-4)}{2}     &  ~~\mbox{if}~~q\equiv 0 \pmod 4, \\
\frac{(q-3)(q-5)}{2}     &  ~~\mbox{if}~~q\equiv 1 \pmod 4, \\
\frac{(q-2)(q-6)}{2}     &  ~~\mbox{if}~~q\equiv 2 \pmod 4, \\
\frac{(q-3)(q-7)}{2}     &  ~~\mbox{if}~~q\equiv 3 \pmod 4,
\end{align*}
for all $q \ge 6$.
\end{corollary}

\begin{proof}
Suppose $q \ge 6$.
\begin{itemize}
\item If $q\equiv 0 \pmod 4$ then $\frac{q-2}{2}$ is odd, and by \cite[Lemma 2.2]{Mitchell24}
    there exists an $\mathcal{OS}_{\frac{q-2}{2}}(2)$ with period
    $\frac{q-2}{2}(\frac{q-2}{2}-1)/2$.  So by Theorem~\ref{theorem:goodify_specialconcatenate}
    there exists a good $\mathcal{SOS}_q(2)$ of period
    $4\frac{q-2}{2}(\frac{q-2}{2}-1)/2=\frac{(q-2)(q-4)}{2}$.

\item If $q\equiv 1 \pmod 4$ then $\frac{q-3}{2}$ is odd, and by \cite[Lemma 2.2]{Mitchell24}
    there exists an $\mathcal{OS}_{\frac{q-3}{2}}(2)$ with period
    $\frac{q-3}{2}(\frac{q-3}{2}-1)/2$. So by Theorem~\ref{theorem:goodify_specialconcatenate}
    there exists a good $\mathcal{SOS}_q(2)$ of period
    $4\frac{q-3}{2}(\frac{q-3}{2}-1)/2=\frac{(q-3)(q-5)}{2}$.

\item If $q\equiv 2 \pmod 4$ then $\frac{q-2}{2}$ is even, and by \cite[Lemma 2.2]{Mitchell24}
    there exists an $\mathcal{OS}_{\frac{q-2}{2}}(2)$ with period
    $\frac{q-2}{2}(\frac{q-2}{2}-2)/2$. So by Theorem~\ref{theorem:goodify_specialconcatenate}
    there exists a good $\mathcal{SOS}_q(2)$ of period
    $4\frac{q-2}{2}(\frac{q-2}{2}-2)/2=\frac{(q-2)(q-6)}{2}$.

\item If $q\equiv 3 \pmod 4$ then $\frac{q-3}{2}$ is even, and by \cite[Lemma 2.2]{Mitchell24}
    there exists an $\mathcal{OS}_{\frac{q-3}{2}}(2)$ with period
    $\frac{q-3}{2}(\frac{q-3}{2}-2)/2$. So by Theorem~\ref{theorem:goodify_specialconcatenate}
    there exists a good $\mathcal{SOS}_q(2)$ of period
    $4\frac{q-3}{2}(\frac{q-3}{2}-2)/2=\frac{(q-3)(q-7)}{2}$.
\end{itemize}
\end{proof}

\subsection{Adjusting the weight}

Just as was the case in the previous section, we need to modify the sequences we have just
constructed to ensure the result has weight coprime to $q'$.  We can employ an identical strategy
to that described in Section~\ref{subsection:adjusting_the_weight}.

\begin{construction}  \label{construction:goodify_specialconcatenate_truncated}
Suppose $q>4$ and $q'=2q+2$ or $q'=2q+3$. Set $x=0$, $y=1$ and $z=q-1$ (which are distinct since
$q>2$). By inspection it also holds that $(x+q+1)+y+z=2q+1$ is coprime to $q'$ for all possible
choices of $q$ and $q'$.

Suppose $S$ is an $\mathcal{OS}_q(2)$ of maximal period $m$ (i.e.\ of period $q(q-1)/2$ ($q$ odd)
or $q(q-2)/2$ ($q$ even)), such that its ring sequence has the form $[xyzx\ldots]$, which exists
from Lemma~\ref{lemma:special_OSq2} --- also observing that since $q>4$ the sequence will contain
at least two occurrences of $x$. Construct $U$ from $S$ using the method of
Corollary~\ref{corollary:specialconcatenate}, and $U'$ from $U$ using
Construction~\ref{construction:goodify_specialconcatenate}, ensuring that the first two zeros in
$U$ are changed to $q+1$. Observe that, from the method of construction, the ring sequence for $U'$
has the form $[q+1,1,q-1,q+1,\ldots]$. Finally, construct $U^{**}$ from $U'$ by deleting the first
three elements of its ring sequence.
\end{construction}

\begin{theorem} \label{theorem:goodify_specialconcatenate_truncated}
Suppose $q>4$ and $q'=2q+2$ or $q'=2q+3$.  If $U^{**}$ is constructed according to the method of
Construction~\ref{construction:goodify_specialconcatenate_truncated} then it is a good
$\mathcal{SOS}_{q'}(2)$ of period $2q(q-1)-3$ ($q$ odd) or $2q(q-2)-3$ ($q$ even) where in every
case $w_{q'}(U^{**})$ is coprime to $q'$.
\end{theorem}

\begin{proof}
By Theorem~\ref{theorem:goodify_specialconcatenate}, the sequence $U'$ is an
$\mathcal{SOS}_{q'}(2)$ of period $2q(q-1)$ ($q$ odd) or $2q(q-1)$ ($q$ odd) where $w_{q'}(U)=0$.
The result now follows immediately by observing that constructing $U^{**}$ from $U'$ does not add
any new 2-tuples, that $w_{q'}(U^{**})=q'-(2q+1)$, and, $q'-(2q+1)$ is coprime to $q'$.
\end{proof}

The following brief example shows the operation of this construction.

\begin{example}  \label{example:weight_adjusted_SOS_12_2}
Suppose $q=5$, $q'=12$ and $n=2$.  We need an $\mathcal{OS}_5(2)$ with ring sequence starting
$[0140\ldots]$.  One possibility is $[01402~13423]$.  As in the previous examples we have
\[ S''=S'||-S' = [0,1,4,0,2,~1,3,4,2,3,~0,11,8,0,10,~11,9,8,10,9]. \]
Analogously to Example~\ref{example:T_and_-T} we have
\[ T' = [7,1,8,5,10,~1,9,4,10,3,~5,11,4,7,2,~11,3,8,2,9]. \]
We simply concatenate them to obtain
\begin{align*}
U = & [0,1,4,0,2,~1,3,4,2,3,~0,11,8,0,10,~11,9,8,10,9,\\
    & ~7,1,8,5,10,~1,9,4,10,3,~5,11,4,7,2,~11,3,8,2,9]
\end{align*}
which by Corollary~\ref{corollary:specialconcatenate} is an $\mathcal{SOS}_{12}(2)$ (and we can
observe it contains no occurrences of $q+1=q'-q-1=6$).

Next, we replace every 0 with $q+1=6$ (since in this case $q'-q-1=q+1$) to obtain
\begin{align*}
U' = & [6,1,4,6,2,~1,3,4,2,3,~6,11,8,6,10,~11,9,8,10,9,\\
     & ~7,1,8,5,10,~1,9,4,10,3,~5,11,4,7,2,~11,3,8,2,9]
\end{align*}
which by Theorem~\ref{theorem:goodify_specialconcatenate} is a good $\mathcal{SOS}_{12}(2)$ with
$w_{12}(U')=0$.

Finally we simply delete the first three terms of $U'$ to obtain
\begin{align*}
U^{**} = & [6,2,~1,3,4,2,3,~6,11,8,6,10,~11,9,8,10,9, \\
         & ~7,1,8,5,10,~1,9,4,10,3,~5,11,4,7,2,~11,3,8,2,9].
\end{align*}
which is a good $\mathcal{SOS}_{12}(2)$ of period 37 with $w_{12}(U^{**})=1$.
\end{example}

The following result follows immediately from Corollary~\ref{corollary:goodSOSq2} and
Theorem~\ref{theorem:goodify_specialconcatenate_truncated}.

\begin{corollary}  \label{corollary:goodSOSq2_unit_weight}
There exists a good $\mathcal{SOS}_{q}(2)$ of period:
\begin{align*}
\frac{(q-2)(q-4)}{2}-3     &  ~~\mbox{if}~~q\equiv 0 \pmod 4, \\
\frac{(q-3)(q-5)}{2}-3     &  ~~\mbox{if}~~q\equiv 1 \pmod 4, \\
\frac{(q-2)(q-6)}{2}-3     &  ~~\mbox{if}~~q\equiv 2 \pmod 4, \\
\frac{(q-3)(q-7)}{2}-3     &  ~~\mbox{if}~~q\equiv 3 \pmod 4,
\end{align*}
for all $q \ge 12$, where in every case the weight of the sequence is a unit modulo $q$.
\end{corollary}

\section{Constructing orientable sequences}

We now consider how to obtain large period orientable sequences using the special orientable
sequences we have constructed earlier in this paper.  We follow two different approaches, both
employing recursive construction methods described in Alhakim et al.\ \cite{Alhakim24}.

\subsection{Special orientable sequences for \texorpdfstring{$n=3$}{n=3}}

We first show how to generate an $\mathcal{SOS}_q(3)$ with large period for arbitrary $q>2$.  We do
so using the following result\footnote{Note that this is actually a special case of the result from
\cite{Alhakim24}.}.  In this case we do not require the input sequences to be good.

\begin{theorem}[\cite{Alhakim24}, Theorem 6.11] \label{theorem:Lempel-special-orientable}
Suppose $S=(s_i)$ is an $\mathcal{SOS}_q(n)$ of period $m$ and $q>2$. If $w_q(S)$ is coprime to $q$
then $D^{-1}(S)$ is an $\mathcal{SOS}_q(n+1)$ of period $qm$ (where $D$ is as defined in
Section~\ref{subsection:dB_graph_and_Lempel}).
\end{theorem}

Combining this with Construction~\ref{construction:specialconcatenate_truncated} and
Theorem~\ref{theorem:specialconcatenate_truncated} we get the following corollary.

\begin{corollary}
Suppose $q\geq 5$ and let $S$ be an $\mathcal{OS}_q(2)$ of maximal period $m$ (i.e.\ of period
$q(q-1)/2$ ($q$ odd) or $q(q-2)/2$ ($q$ even)), such that its ring sequence has the form
$[xyzx\ldots]$ or, should $x \not=0$, the form $[0xyzx\ldots]$, which exists from
Lemma~\ref{lemma:special_OSq2}, where $x$, $y$ and $z$ are as specified in
Construction~\ref{construction:specialconcatenate_truncated}. Suppose $U^{*}$ is constructed from
$S$ using the method of Construction~\ref{construction:specialconcatenate_truncated}, where
$q'=2q+1$ or $q'=2q+2$. Then $D^{-1}(U^{*})$ is a $\mathcal{SOS}_{q'}(3)$ of period $2q^3-2q^2-3q$
($q$ odd) or $2q^3-4q^2-3q$ ($q$ even).
\end{corollary}

\begin{proof}
By Theorem~\ref{theorem:specialconcatenate_truncated}, $U^{*}$ is a $\mathcal{SOS}_{q'}(2)$ of
period $2q(q-1)-3$ ($q$ odd) or $2q(q-2)-3$ ($q$ even) where $w_{q'}$ is coprime to $q'$.  The
result follows from Theorem~\ref{theorem:Lempel-special-orientable}.
\end{proof}

We also have the following, which is immediate from Corollary~\ref{corollary:SOSq2_truncated}.

\begin{corollary}    \label{corollary:SOSn3}
There exists an $\mathcal{SOS}_{q}(3)$ of period:
\begin{align*}
\frac{q^3-6q^2+2q}{2}     &  ~~\mbox{if}~~q\equiv 0 \pmod 4, \\
\frac{q^3-6q^2-q}{2}      &  ~~\mbox{if}~~q\equiv 1 \pmod 4, \\
\frac{q^3-8q^2+6q}{2}     &  ~~\mbox{if}~~q\equiv 2 \pmod 4, \\
\frac{q^3-4q^2-3q}{2}     &  ~~\mbox{if}~~q\equiv 3 \pmod 4,
\end{align*}
for all $q \ge 11$.
\end{corollary}

Observe that these sequences have period a little less than the $\mathcal{OS}_{q}(3)$ sequences
constructed in \cite{Mitchell24}.  However, the sequences constructed here have the additional
property of being both orientable and negative orientable, which may be of use in some
applications.

\subsection{Special orientable sequences for general \texorpdfstring{$n$}{n}}

We next show how to construct $\mathcal{SOS}_q(n)$ with large period for arbitrary $q>3$ and
arbitrary $n>2$.   We employ the following result\footnote{As above, this is actually a special
case of the result from \cite{Alhakim24}.}.  Note that in this case we \emph{do} require our input
sequences to be good.

\begin{theorem}[\cite{Alhakim24}, Corollary 6.22]  \label{theorem:Corollary_6.22}
Suppose $S_n$ is a good $\mathcal{SOS}_q(n)$ of period $m_n$, where $w_q(S_n)$ is coprime to $q$.
Recursively define the sequences $S_{i+1}=\mathcal{E}_a(D^{-1}(S_i))$, where $a =
1-w_q(D^{-1}(S_i))$, for $i\geq n$, and suppose $S_i$ has period $m_i$ ($i>n$). Then, $S_{i}$ is an
$\mathcal{SOS}_q(i)$ for every $i\ge n$, and $m_{n+j}=qm_{n+j-1}+1$ for every $j\geq 1$ (and hence
$m_{n+j}=q^{j}m_n+\frac{q^j-1}{q-1}$ for every $j\geq 1$).
\end{theorem}

Combining this theorem with Construction~\ref{construction:goodify_specialconcatenate_truncated}
and Theorem~\ref{theorem:goodify_specialconcatenate_truncated} we get the following corollary.

\begin{corollary}  \label{corollary:good_special_general_n}
Suppose $q\geq 5$ and let $S$ be an $\mathcal{OS}_q(2)$ of maximal period $m$ (i.e.\ of period
$q(q-1)/2$ ($q$ odd) or $q(q-2)/2$ ($q$ even)), such that its ring sequence has the form
$[xyzx\ldots]$ or, should $x\not=0$, the form [$0xyzx\ldots]$, which exists from
Lemma~\ref{lemma:special_OSq2}, where $x$, $y$ and $z$ are as specified in
Construction~\ref{construction:goodify_specialconcatenate_truncated}. Suppose $U^{**}$ is
constructed from $S$ using the method of
Construction~\ref{construction:goodify_specialconcatenate_truncated}, where $q'=2q+2$ or $q'=2q+3$.
Setting $S_2=U^{**}$ in Theorem~\ref{theorem:Corollary_6.22}, $S_n$ is a good
$\mathcal{SOS}_{q'}(n)$ of period
\begin{align*}
2q'^{n-2}(q(q-1)-3)+\frac{q'^{n-2}-1}{q'-1} &~~~(q\text{~odd), or}\\
2q'^{n-2}(q(q-2)-3)+\frac{q'^{n-2}-1}{q'-1} &~~~(q\text{~even})
\end{align*}
for every $i\geq 2$.
\end{corollary}

\begin{proof}
By Theorem~\ref{theorem:goodify_specialconcatenate_truncated}, $U^{**}$ is a good
$\mathcal{SOS}_{q'}(2)$ of period $2q(q-1)-3$ ($q$ odd) or $2q(q-2)-3$ ($q$ even), where $w_{q'}$
is coprime to $q'$.  The result follows from Theorem~\ref{theorem:Corollary_6.22}.
\end{proof}

\begin{corollary}    \label{corollary:SOS_generaln} There exists an $\mathcal{SOS}_{q}(n)$ of period:

\begin{align*}
\frac{q^n-6q^{n-1}+2q^{n-2}}{2}+\frac{q^{n-2}-1}{q-1} &  ~~\mbox{if}~~q\equiv 0 \pmod 4, \\
\frac{q^n-8q^{n-1}+9q^{n-2}}{2}+\frac{q^{n-2}-1}{q-1} &  ~~\mbox{if}~~q\equiv 1 \pmod 4, \\
\frac{q^n-8q^{n-1}+6q^{n-2}}{2}+\frac{q^{n-2}-1}{q-1} &  ~~\mbox{if}~~q\equiv 2 \pmod 4, \\
\frac{q^n-10q^{n-1}+15q^{n-2}}{2}+\frac{q^{n-2}-1}{q-1} &  ~~\mbox{if}~~q\equiv 3 \pmod 4,
\end{align*}
for all $q \ge 12$, and $n\geq 2$.
\end{corollary}

\begin{proof} Suppose $q \ge 12$ and $n\geq 2$.
\begin{itemize}
\item If $q\equiv 0 \pmod 4$ then $r=\frac{q-2}{2}$ is odd, and by
    Corollary~\ref{corollary:good_special_general_n} there exists a good $\mathcal{SOS}_q(n)$
    with period $2q^{n-2}(r(r-1)-3)+\frac{q^{n-2}-1}{q-1}$.  Substituting in $r=(q-2)/2$ the
    result follows.

\item If $q\equiv 1 \pmod 4$ then $r=\frac{q-3}{2}$ is odd, and by
    Corollary~\ref{corollary:good_special_general_n} there exists a good $\mathcal{SOS}_q(n)$
    with period $2q^{n-2}(r(r-1)-3)+\frac{q^{n-2}-1}{q-1}$.  Substituting in $r=(q-3)/2$ the
    result follows.

\item If $q\equiv 2 \pmod 4$ then $r=\frac{q-2}{2}$ is even, and by
    Corollary~\ref{corollary:good_special_general_n} there exists a good $\mathcal{SOS}_q(n)$
    with period $2q^{n-2}(r(r-2)-3)+\frac{q^{n-2}-1}{q-1}$.  Substituting in $r=(q-2)/2$ the
    result follows.

\item If $q\equiv 3 \pmod 4$ then $r=\frac{q-3}{2}$ is even, and by
    Corollary~\ref{corollary:good_special_general_n} there exists a good $\mathcal{SOS}_q(n)$
    with period $2q^{n-2}(r(r-2)-3)+\frac{q^{n-2}-1}{q-1}$.  Substituting in $r=(q-3)/2$ the
    result follows.
\end{itemize}
\end{proof}

\section{Concluding remarks}

In this paper we have constructed orientable sequences with the additional property that they are
also negative orientable. We used an approach proposed in \cite{Alhakim24a} to generate orientable
sequences with large period of any order over an alphabet of any size using 'starter' sequences
with this additional property. Whilst this yields sequences with  shorter periods than general
orientable sequences, the periods remain asymptotic to the optimal as the alphabet size increases
and the additional property could be a benefit in some applications.

It remains an open problem to find constructions of orientable sequences with optimal periods.

\providecommand{\bysame}{\leavevmode\hbox to3em{\hrulefill}\thinspace}
\providecommand{\MR}{\relax\ifhmode\unskip\space\fi MR }
\providecommand{\MRhref}[2]{%
  \href{http://www.ams.org/mathscinet-getitem?mr=#1}{#2}
} \providecommand{\href}[2]{#2}

\end{document}